\documentclass{amsart}
\usepackage{latexsym,amsxtra,amscd,ifthen}
\usepackage{amsfonts}
\usepackage{verbatim}
\usepackage{amsmath}
\usepackage{amsthm}
\usepackage{amssymb}

\numberwithin{equation}{section}

\theoremstyle{plain}
\newtheorem{theorem}{Theorem}[section]
\newtheorem{theorema}[theorem]{Theorem}

\newtheorem{prop}[theorem]{Proposition}
\newtheorem{lemma}[theorem]{Lemma}
\newtheorem{cor}[theorem]{Corollary}

\theoremstyle{definition}

\newtheorem{definition}[theorem]{Definition}

\newtheorem{exam}[theorem]{Example}

\newtheorem{rema}[theorem]{Remark}

\newcommand{\N}{\ensuremath{\mathbb N}}
\newcommand{\Q}{\ensuremath{\mathbb Q}}

\newcommand{\R}{\ensuremath{\mathbb R}}
\newcommand{\Z}{\ensuremath{\mathbb Z}}

\def\la{\langle}
\def\ra{\rangle}

\def\Card{\mbox{Card\,}}

\def\Char{\mbox{char\,}}

\def\ext{\mbox{Ext\,}}

\def\gr{\mbox{\bf gr\,}}

\def\id{\mbox{id\,}}

\begin{document}

\title{A remark on Golod--Shafarevich algebras}
\thanks{Partially
supported by the grant 14-01-00416 of the Russian Basic Research Foundation 
}%

\author{Dmitri Piontkovski}

      \address{Department of High Mathematics for Economics,
 Myasnitskaya str. 20, National Research University `Higher School of Economics', Moscow  101000, Russia
}

\email{piont@mccme.ru}



\date{\today}

\begin{abstract}
We show that
a direct limit of surjections of (weak) Golod--Shafarevich  algebras is a weak Golod--Shafarevich algebra as well. This holds both for graded and for filtered algebras provided that the filtrations are induced by the filtration of the first entry of the sequence. 
It follows that the limit is an algebra of exponential growth. An example shows that the assumptions of this theorem cannot be directly weakened. 
\end{abstract}

\maketitle

\section{Introduction}

We consider finitely generated algebras and completed finitely generated algebras over a fixed field $k$.
Given such an algebra $A$, let us define a descending $\R_{>0}$--filtration on it 
in the following way. We choose a discrete set of generators $X$
of the algebra $A$ and a degree function $\deg :X \to \R_{>0}$. Then for each nonzero element $x\in A$ we say $\deg x \ge \alpha$ iff $x = \sum_i c_i v_i$, 
where $0 \ne c_i \in k$ and $v_i $ is a product of generators of the total degree at least $\alpha$. 
In addition, we put $\deg 0 = +\infty$. This defines the degree for each element of the algebra.

 The degree function $\deg :X \to \R_{>0}$ defines the descending filtrations on the free algebra $k\la X \ra$
 and the completed free algebra $k\la \la X \ra \ra$. Hence we  uniquely define the degrees for 
 the relations $R = \{r_1, r_2, \dots \}$ of the algebra $A$ in the both finitely generated and complete cases.  
 Also, one can define a {\em Hilbert series} of the algebra $A$
 as a generalized formal power series 
 $$
 H_A (z) = \sum_{\alpha \ge 0} z^\alpha \dim (A_\alpha/\cup_{\beta>\alpha}A_\beta) , 
 $$ 
 where $A_{\alpha}$ is the vector space of elements 
 of degree $\ge \alpha$.  
 There is only finite number of monomials on the generators of degree $\le  \alpha$ for each $\alpha$, 
 so, the sum always has sense.  Then $ H_A (z)$ a Hahn series with the value group $G$ which is the subgroup of $(\R,+)$ generated by the degrees of the elements of $X$.
 
 Given a discrete subset $S$ of a filtered algebra such that for each $\alpha > 0$ the subset of elements 
 of degree $\alpha$ is finite, put also $H_S(z) = \sum_{s\in S} z^{\deg s}$.  

Following Golod and Shafarevich~\cite{gsh}, we associate to $A$ a generalized formal power series (or a polynomial, in the case of finite set of relations $R$)
$f_A(z) := 1-H_X(z) +H_R(z)$, where $R$ is a discrete set of the defining relations of $A$.
 
 \begin{definition}
 \label{def:gs}
 A filtered algebra $A$ as above is called  Golod--Shafarevich (respectively, weak { Golod--Shafarevich})  if for some $z_0\in [0,1]$
 the series $f_A(z)$ converges 
and $f_A(z_0) < 0$ (respectively, $f_A(z_0) \le 0$).
%
  \end{definition}
  
  We write wGS and GS for (weak) Golod--Shafarevich. 
  Our notion of GS algebra is analogous to the notion of generalized GS group~\cite[4.2]{e}.
  
  Note that if a (w)GS algebra $A=\{ X | R\}$ has another presentation $A=\{ X | R'\}$ with $X'\subset X$ and $R' \subset R$ 
  then it is still (w)GS  w.~r.~t. this new presentation. 
  However, we will see in~Example~\ref{exam_add_new_var} that in some cases
 the properties of  being (w)GS valuably depend on the choice of generators and relations.

In addition, the next example shows 
 that the properties of algebra of being (w)GS essentially depend on the degree function.  
 
 \begin{exam}
Consider
 an algebra $M$ generated by three variables $x,y,z$ 
 subject to the set of monomial relations $R=\{ x^2,y^2,xz \} $. If we put 
 $\deg x = \deg y = \deg z =1$ (standard grading), the value of $f_M(z) = 1-3z+3z^2 $  is positive for all positive $z$, so that $M$ is neither GS nor wGS. In contrast, if we consider another grading by putting $\deg x = \deg y = 2$ and $ \deg z =1 $, we get $f_M(z) = 1-z-2z^2+ z^3+2z^4$.
 Then $f_M(0.6)<0$, so that $M$ is GS.
 
 Note that the polynomial ring on three variables and, more general, a quadratic Artin--Shelter regular ring of dimension three has the same number and degrees of generators and relations as the algebra $M$ above. However, these rings are neither GS nor wGR w.~r.~t. any degree function because, by definition, their growth is polynomial and  for such a ring $A$ we have $f_A(1) = 1$. It shows also that the property of being it essentially depend not only on the number of relations but also on their form.
 \end{exam}
 
 The main reason to consider (w)GS algebras is the 
 famous Golod--Shafarevich theorem~\cite{gsh}. We use it   in its   form for filtered algebras
 (see~\cite[Prop.~2.7 and Th.~4.1]{e}), that is, in a modified  Vinberg's form~\cite{vin}.
 
 \begin{theorem}
 In the above notation, suppose that the degree function $\deg$ is integer-valued. Then there is a coefficient-wise inequality of formal power series 
 $$
      \frac{f_A(z)}{1-z} H_A(z) \ge \frac{1}{1-z}.
 $$
 \end{theorem}

The following corollary is also sometimes called Golod--Shafarevich theorem. Note that we do not assume here that the degree function $\deg$ is integer-valued.

\begin{theorema}
\label{th:gsII}
The inequality  $f_A(z_0) H_A(z_0) \ge 1$ holds for all real $z_0\in [0,1]$ (where we assume $+\infty > 1$). 
\end{theorema}

We deduce this version of GS theorem from the above one in Section~\ref{sec:def_gs} below.
Also, we recall  there a proof of the following corollary (which is as well standard, at least, in its essential parts). 
 
 \begin{cor}
 \label{cor:growth_GS}
 All wGS algebras are infinite-dimensional. 
 Moreover, if either $A$ is a GS algebra  or the set of relation $R$
 is infinite then the associated graded algebra $\gr A$
 has exponential growth.
 \end{cor}
 
 
The analogous notion of (generalized) Golod--Shafarevich group 
is widely 
used in various  constructions 
of groups with exotic  properties~\cite{e}. One of recent constructions is a so-called ``Ershov--Jaikin  Monster'', that is, an example  of an infinite finitely generated residually finite
$p$-torsion group in which every infinite finitely generated subgroup has finite index. 
This group is constructed as a direct limit of groups having  
generalized Golod--Shafarevich subgroups of finite index~\cite[Th.~1.5]{ej} 
(see also~\cite[Sec.~13]{e}). However, it is unknown 
if such a monster itself can be generalized Golod--Shafarevich or not. 

Motivated  by a question of Efim Zelmanov, we discuss here a direct limit of (w)GS algebras. 
Let $\phi: A\to B$ be a surjection of algebras, and suppose that $X$ and $Y= \phi(X)$ be some generating sets of $A$ and $B$. If $B = k\la X| R\ra$ is (w)GS then it is easy to see 
that $A$ is as well (w)GS w.~r.~t. some presentation $A = k\la Y |S \ra $  such that $\tilde \phi(R) \subset S$  and some degree function satisfying   $\deg \phi(x) = \deg x$ 
 for all $x\in X$, where $\tilde \phi$ is an extension of $\phi$ to the morphism of free algebras
  on $X$ and $Y$. We call here such degree functions and presentations of algebras $A$ and $B$ {\em compatible}.

The next theorem is the main result of this note. 

\begin{theorema}[Theorem~\ref{th:1}, Corollary~\ref{cor:1}]
\label{th:main}
Let $A_0 \to A_1 \to \dots$ be a sequence of surjections of algebras which are wGS with respect to compatible presentations and degree functions.  Then the limit
$$
A = \lim_{n\to \infty} A_n
$$
is wGS with respect to the induced degree function. 
Moreover, $A$ has exponential growth provided that almost all algebras $A_n$ are GS w.~r.~t. the above degree function. 
\end{theorema}

We have observed in Theorem~\ref{th:main} that the na\"ive way to transfer the construction of Ershov and Jaikin 
does not lead to Smoktunowicz-type examples of (nil) algebras with exotic properties, if we 
try to obtain new algebras as a direct limits of GS ones. Note also that the above construction of Ershov and Jaikin and  is based on transferring the GS property froma group to its subgroups of finite index. It seems that there is no analogous statements for algebras~\cite{voden}.   

The assumption of Theorem~\ref{th:main} cannot be weakened in any obvious way.
For instance, if we just assume that  the algebras $A_n$ are either GS w.~r.~t. arbitrary degree functions, or have exponential growth, or contain free nonabelian subalgebras, then the algebra $A$ can be non-wGS. This is shown in  Example~\ref{ex2}. If we change the degree function while transferring to the next element of the sequence, the limit may occur far from being GS.  

\subsection*{Acknowledgement}

This text 
is motivated 
  by a question of Efim Zelmanov stated in his talk in the 2012 workshop 
   {\em Golod-Shafarevich Groups and Algebras
and Rank Gradient}.  I am grateful to him and to Mikhail Ershov for the possibility
to participate in the workshop. I am grateful also to ESI, Vienna, for the hospitality during it. My special thanks to Andrei Jaikin for valuable remarks and for dramatical simplification of my original arguments.


\section{On the definition of  Golod--Shafarevich algebras}
\label{sec:def_gs}

Given a set $X = \{ x_1, x_2, \dots \}$ and a function $\deg:X\to \R_{\ge 0}$,
let us denote by $\widetilde X = \{ \widetilde x_1, \widetilde x_2, \dots \}$
a set of formal variables to which we assign the same degrees $\deg \tilde x_i = \deg x_i$.
Let $A$ be an algebra generated by its subset $X$ subject to the relations $R \subset F$, where either $F =  k\la  \widetilde X \ra $ or $F =  k\la \la  \widetilde X \ra \ra $. 

First we observe that it is sufficient to check the GS condition under the assumptions that the sets $X$ and $R$ are minimal in some sense. 

\begin{prop}
\label{prop:min4gs}
Suppose that an algebra $A = k\la X | R \ra$ as above is (w)GS w.~r.~t. a degree function 
$\deg: X\to \R_{\ge 0}$. 

a. Let  $R' \subset R$ be any subset. Then the algebra $B = k\la X | R' \ra$ is (w)GS w.~r.~t.
the same degree function on $X$. For example, if $R'$ generates the same ideal in $F$ as $R$,
then $A$ is (w)GS w.~r.~t. the presentation $A = k\la X | R' \ra$.

b. Let $X' \subset X$ be a subset generated the same algebra $A$, and let $F' =  k\la  \widetilde {X'} \ra $ or, respectively, $F' =  k\la \la  \widetilde X' \ra \ra $. Suppose that $R = R' \cup \bar R$,
where $R' \subset F'$ and $\bar R = \{ r_x| x\in X \setminus X' \}$ with $r_x = x- f_x, f_x \in F'$.
Then $A$ is (w)GS w.~r.~t. the presentation $A = k\la X' | R' \ra$ for the degree function  $\deg' : X' \to  \R_{\ge 0}$ which is a restriction of $\deg: X\to \R_{\ge 0}$.
\end{prop}

\begin{proof}
a. It follows immediately from Definition~\ref{def:gs} that $f_B(z_0 ) \le f_A(z_0)$ for all $z_0\in [0,1]$.
So,  $B$ is (w)GS provided that $A$ is.

b. 
 Let $f(z)$ and $\bar f(z)$ be the series $f_A(z)$ for the presentations $A = k\la X | R \ra$ and $A = k\la X' | R' \ra$, respectively. It is sufficient to show that $\bar f(z_0) \le f(z_0)$ for all $z_0\in [0,1]$.
 By induction, we can assume that the set  $X \setminus X'$ consists of a single element $x$.
 
 Let $d = \deg x$ and $d' = \deg' f(x)$. We have 
 $$
 f(z) = 1 - H_{X'}(z) - z^{d} +z^{\deg r_x} + \sum_{r\in R'} z^{\deg r}.  
 $$
 Here  $\deg'r_x = \min\{d,d' \}$, so that for all  $z_0 \in [0,1]$ we have 
 $z_0^{\deg' r_x} - z_0^{d} \ge 0$. 
 Note that for each $y\in F$ we have $\deg y \le deg' y$. 
 Thus, for each $z_0 \in [0,1]$ 
 $$
  f(z_0) \le 1 - H_{X'}(z_0) +\sum_{r\in R'} z_0^{\deg' r} = \bar f(z_0)
 $$ 
\end{proof}

\begin{exam}
\label{exam_add_new_var}
Note that the condition $R'\subset F'$ is essential in Proposition~\ref{prop:min4gs}b.
Indeed, 
consider an algebra $A$ with the minimal presentation $A =k\la X' | R' \ra$, where 
 $X' = \{ x,y \}$, $R' = \{ (xy)^2x, (xy)^2y, x(xy)^2 \}$.
This algebra is not (w)GS w.~r.~t. the standard degree function $\deg x= \deg y =1$
(since the polynomial $f_A(z) = 1-2z+3z^5$ has no positive roots). Still, if we add a new
variable $v = xy$ and put $\deg v = 3$, for the non-minimal presentation
$$A = k\la x,y,v | v - xy, v^2x, v^2y, xv^2 \ra
$$ 
we have 
$$
f_A(z) = 1-2z-z^3+z^2+3z^7.
$$
For this presentation $A$ is GS, since $f_A(0.7)<0$.
\end{exam}


\begin{proof}[Proof of Theorem~\ref{th:gsII}]
Let $R' \subset R$ be a subset of $R$ which minimally generates the same ideal of relations of $A$. Then $f(z) = 1-H_X(z)+H_R(z) \ge 1-H_X(z)+H_{R'}(z)$
 for all positive $z$. So, it is sufficient to prove Theorem~\ref{th:gsII} in the case $R = R'$. In this case, it is easy to see that 
 the series $H_R(z_0)$ converges for some positive $z_0$ provided that the series $H_A(z_0)$
 converges. It follows that  the series $f_A(z_0)$ converges too.

Let us show that $f_A(z_0) H_A(z_0) \ge 1$ for all real $z_0\in [0,1]$ for which the
 series $ H_A(z)$ converges. If the degree function is integer-valued, this is obvious from Golod--Shafarevich theorem. If the degree function is rational-valued, then the Golod--Shafarevich theorem gives a coefficient-wise inequality 
$$
      \frac{f_A(z)}{1-z^{1/q}} H_A(z) \ge \frac{1}{1-z^{1/q}}
 $$ 
 in $\Q[[z^{1/q}]]$, where $q$ is a common denominator of the degrees of the generators,  so that that $f_A(z_0) H_A(z_0) \ge 1$. 
 Now, it remains to observe that
 given $\varepsilon >0$, any degree function can be replaced by a  rational-valued one such that the maximal 
 value of the product $f_A(z_0) H_A(z_0)$ for $z_0\in [0,1]$  is changed less than by $\varepsilon$.
\end{proof}

\begin{proof}[Proof of Corollary~\ref{cor:growth_GS}]
 Note that if an algebra $A$ is  (w)GS 
 w.~r.~t. some choice of the generators, the relations, and the degree function, 
 then it should remain to be w(GS) if we replace the set $R$
 by any its subset which minimally generates the same ideal. So, we will assume that the set $R$ minimally generates the ideal of relations of $A$. As above, it follows that 
 the series  $f_A(z_0)$ converges for some positive $z_0$ provided that the series $H_A(z_0)$
 converges.

 Let $D >0$ be the maximal value of the degrees of the generators $x\in X$. Then 
 the  words on $X$ of length  at most $n$  have degree at most $Dn$, so that 
 the number $h_n$ of linearly independent ones is not less then $\dim (A/A_{>nd})$.
 The last number is the sum of the coefficients $a_\gamma$ of the series $H_A(z)$ with $\gamma \le nd$. Let $H(z) = \sum_{n\ge 0} h_n z^n$ be the Hilbert series of $A$ w.~r.~t. the standard filtration. Then we have the inequalities of series (under the assumption $z\in (0,1)$): 
$$
H_A(z^{1/D}) = \sum_{\gamma>0} a_\gamma z^{\gamma/D}
= 1+ \sum_{n\ge 0} \sum_{nD< \gamma \le (n+1) D} a_\gamma z^{\gamma/D}
$$
$$
\le 1+ \sum_{n\ge 0} z^n \sum_{nD< \gamma \le (n+1) D} a_\gamma  \le 
1+ \sum_{n>0} (h_{n+1}-h_{n})z^n = \frac{1-z}{z}(H(z)-1).
$$ 
 
 Now, suppose that $z_0$ is the minimal positive real number such that  $f(z_0) = 0$. 
 Then the series $H_A(z_0)$ diverges. 
If $z_0<1$, it follows from the above inequalities that 
 the formal power series $H(z)$ diverges at $z = z_0^{1/D} <1$, so that its coefficients $h_n$ grow exponentially.  So, one can assume that $z_0=1$ and
 $f_A(z) > 0$ for all $z\in [0,1)$. It follows that 
 $A$ is not GS (whereas $A$ is still wGS) and
 $$
        f_A(1) =  1- \Card X +  \Card R = 0,
 $$
 that is, the set $R$ is finite.
 \end{proof}

\section{Golod--Shafarevich algebras}
\label{sec:res}

Let $B = k \la X|R\ra$ be an algebra with a degree function $\deg$ (induced by a degree function on the (completed) free algebra on the finite set $X$), and let $\phi: B\to C$
be a surjection of filtered algebras, where $C = \la \phi(X) | R' \ra $.
The surjection $\phi$ induces a unique degree function on $C$ with respect to the generating set $\phi(X)$.

\begin{theorema}
\label{th:1}
Let $\{A_i\}_{i\ge 0} $ be a sequence of wGS algebras with surjective maps $\phi_i:A_i \to A_{i+1}$ with the degree function on $A_i$ induced by the degree function on $A_0$, 
and let 
$A= \lim A_i$  be its direct limit. Then $A$ is wGS with respect to the induced degree function.
\end{theorema}

\begin{proof}
Let  $A_i = \la X_i|R_i \ra$ be presentations of the algebras $A_i$ via minimal generators and relations 
such that $X_{i+1} \subset \phi_i(X_i)$, and let  $A=\la X|R \ra$ be the corresponding presentation of $A$
such that $X = \lim X_i$. Since the descending chain $X_1(z) \ge X_2(z) \ge \dots$ of polynomials with positive coefficients $X_i(z)$ should stabilize,  for some $n\ge 0$ we have $X_n(z) = X_{n+1}(z) = \dots = X(z)$. 
So, we can (and will) assume that all algebras $A_1, A_2, \dots $ and $A$ have the same minimal set of generators $X$ such that $\phi|_X = \id$.

Now, each map $\phi_i$ induces an inclusion $(R_i) \subset (R_{i+1}) \subset (R)$ of ideals in the free algebra 
$F= k\la X \ra $ (or, respectively, $F = k\la\la X \ra\ra$ ). Given a subset $S\subset F$, let $S^{(n)} = S \setminus F_{> n}$ denotes the subset of its elements of degree at most $n$. Fix $n\ge 0$. 

 Since $\lim_i (R_i)^{(n)} = (R)^{(n)}$ is a limit of finite-dimensional vector spaces, we have $(R_i)^{(n)} = (R)^{(n)}$ for all $i>>0$. Hence there exists $m = m(n)$
such that $R_i^{(n)} = R^{n}$ for all $i\ge m$. Note that if the sequence $\{ A_i \}$ is not eventually constant then $\lim_{n\to \infty} m(n) = \infty$. 

Let $R'_n = R^{(n)}$.  Then $\lim (R'_n) = (R)$ (where the direct limit is defined by the natural inclusions $R'_n \subset R'_{n+1}$), so that  $A = \lim F/(R_n)$. Moreover, since $R'_n = R_{m(n)}^{(n)}\subset R_{m(n)}$, the algebras $A'_n = F/(R_n)$ are wGS 
 (because of the following trivial corollary of the (w)GS condition: if a quotient $C $ of an algebra $B$ is (w)GS and the minimal set of generators $X$ is the same for the both algebras, then  $B$ is (w)GS as well)\footnote{Analogously, $A'_n$ is GS provided that $A_{m(n)}$ is. This fact is used in Corollary~\ref{cor:1} below.}. 
So, without loss of generality we replace the sequence $\{ A_i \}$ by the sequence $\{ A_n' \}$, 
that is,  
we assume that each $R_n(z)$ is a finite sum of powers of $z^\alpha$  with $\alpha \le n$ such that $R(z) = R_n(z) +o(z^n)$.

Let $f_i(z) = 1-X(z)+R_i(z)$. 
Consider the set $S_i = \{ z \in [0,1] | f_i(z) \le 0 \}$ 
(where we assume $+\infty > 0$). The set $S_i$ is a finite union of closed intervals (moreover,  it follows from the Laguerre theorem~\cite{lag} that 
the number of the intervals is at most $\Card X +1$). Since $R_{i+1}(z) \ge R_(z)$ 
for all $z\in [0,1]$ and all $n\ge 0$, we obtain a descending chain 
$S_0\supset S_1 \supset $ of closed subsets of the compact set $[0,1]$. It has a common point 
$z_0\in [0,1]$. We have $f_i(z_0)\le 0$, so that $f(z_0)=\lim_{i\to +\infty}f_i(z_0)\le 0$. Thus,  $A$ is wGS.
\end{proof}

\begin{cor}
\label{cor:1}
In the notation of Theorem~\ref{th:1}, suppose that almost all algebras $A_i$ are GS.
Then $A$ has exponential growth.
\end{cor}

\begin{proof}
Suppose that $A$ has subexponential growth. Then the series $A(z)$ converges for all $0\le z< 1$. 
By the Golod--Shafarevich inequality $A(z)f_A(z) \ge 1$, we have  $f_A(z)>0$
for all $z\in [0,1)$. 
On the other hand,  Theorem~\ref{th:1} implies that $f_A(z_0)$
converges and is non-positive for some $z_0\in [0,1]$.  
 It follows that $z_0=1$ and  $f_A(1)$ converges to zero, 
that is,  $1-X(1)+R(1) = 0$, where $X(1) = \Card X $ 
and $R(1) = \Card R$. Therefore,  $R$ is finite-dimensional  
and coincides with some $R'_n$, in the notation of the proof of Theorem~\ref{th:1}.
Thus $A = A'_n$ is GS, that is, $f_A(z_0) < 0$ for some $z_0 \in [0,1]$, a contradiction.
\end{proof}

The next example shows that  the assumptions of Theorem~\ref{th:1} cannot be weakened 
in direct way. We see that  the direct limit of a sequence of algebras $A^n$ can be non-wGS  if each $A^n$ is a finite modules over some free subalgebra or  if it is GS for some particular degree function.

\begin{exam}
\label{ex2}
Consider the sequence of algebras $A^n = \la x,y | x^2, xyx, \dots, xy^{n}x\ra $, where 
$n = 0,1,\dots$ 
Let 
$$A = \lim_{n\to \infty} A^n = \la x,y | xy^{t}x , t\ge 0\ra.
$$
We claim that each algebra $A^n$ satisfies the following properties:

(i) it is GS w.~r.~t.  some degree function depending on $n$;

(ii) it is a finite left module over a subalgebra which is free of rank two 
(hence, this subalgebra is GS w.~r.~t. 
the standard degree function induced from $A_0$).

At the same time, the algebra $A$ has linear growth, in particular, $A$ is not wGS with respect to any choice of generators and  degree function. 

\begin{proof}
To prove~(i), let $\deg x = a > 1$ and $\deg y =1$. 
Then for $z\in [0,1]$ we have
$$
f_{A^n}(z)  = 1-z-z^a+z^{2a}(1+z+\dots +z^n) \le g_n(z),
$$
where $g_n(z) = 1-z-z^a +nz^{2a}$. Let  $z_0 = (2n)^{-a^{-1}}$. Then  
 $g_n(z_0) = 1-z_0 -(4n)^{-1}$, so that for $a>>n$ we have
$f_{A^n}(z_0)   \le g_n(z_0)<0$.

To prove~(ii), consider  
the free  subalgebra $C^n = \la y, xy^{n+1}\ra \subset A^n $. Then the algebra $A^n$ as a left $C^n$--module is generated by the finite set $\{ 1,x,xy,\dots, xy^n\}$.  To evaluate the growth of the algebra $A$, just observe that its $t$-th component $A_t$ in standard grading is spanned by the set $\{y^t,xy^{t-1}, y^{t-1}x\}$, so that  $\dim A_t =3$ for each $t\ge 3$.
\end{proof}
 \end{exam}

\end{document}